\documentclass[12pt]{amsart}
\usepackage{latexsym}
\usepackage{amsmath}
\usepackage{amssymb}
\usepackage{amscd}
\usepackage{amsthm}
\usepackage{xypic}
\xyoption{curve}
\usepackage{ifthen}
\usepackage{hyperref}
\usepackage{graphicx}

\theoremstyle{plain}
 \newtheorem{theorem}{Theorem}[section]
 \newtheorem{proposition}[theorem]{Proposition}
 \newtheorem{lemma}[theorem]{Lemma}

\theoremstyle{definition}
 
 \newtheorem{example}[theorem]{Example}

 \newtheorem{remark}[theorem]{Remark}

\newcommand{\PP}{\mathbb{P}}
\newcommand{\QQ}{\mathbb{Q}}
\newcommand{\ZZ}{\mathbb{Z}}
\newcommand{\CC}{\mathbb{C}}

\newcommand{\xycenter}[1]{\begin{center} \mbox{\xymatrix{#1}}\end{center}}
\newcommand{\rmspec}{\mathrm{Spec}\,}

\newboolean{xlabels}
\newcommand{\xlabel}[1]{
                        \label{#1}
                        \ifthenelse{\boolean{xlabels}}
                                   {\marginpar[\hfill{tiny #1}]{{\tiny #1}}}
                                    {}
                          }
\setboolean{xlabels}{false}

\newcommand{\slC} {\mathrm{SL}_3 (\CC )}

\begin{document}
\title[A Clebsch-Gordan formula]{A Clebsch-Gordan formula for $\slC$ and applications to rationality}
\author{Christian B\"ohning}
\author{Hans-Christian Graf v. Bothmer}
\thanks{Supported by the German Research Foundation 
(Deutsche Forschungsgemeinschaft (DFG)) through 
the Institutional Strategy of the University of G\"ottingen}
\maketitle

\begin{abstract}
If $R$, $S$, $T$ are irreducible $\mathrm{SL}_3 (\CC )$-representations, we give an easy and explicit description of a basis of the space
of equivariant maps $R\otimes S \to T$ (Theorem \ref{tClebschGordan}). We apply this method to the rationality problem for invariant function
fields. In particular, we prove the rationality of the moduli space of plane curves of degree $34$. This uses a criterion which ensures the
stable rationality of some quotients of  Grassmannians by an $\mathrm{SL}$-action (Proposition \ref{pStabRatGrassmannians}).
\end{abstract}

\section{Introduction}
For the group $\mathrm{SL}_2 (\CC )$ the irreducible representations are the $V(d):= \mathrm{Sym}^d (\CC^2)$. If $d_1$, $d_2$, $n$
are nonnegative integers such that $0\le n \le \mathrm{min} (d_1 , d_2)$, and if for $f\in V(d_1)$ and $g\in V(d_2)$ one
puts
\begin{gather*}
\psi_n( f,g) := \frac{(d_1-n) !}{d_1 !} \frac{(d_2-n) !}{d_2 !}\sum_{i=0}^n (-1)^i {n \choose i} \frac{\partial^n
f}{ \partial z_1^{n-i}\partial z_2^i} \frac{\partial^n g}{\partial z_1^i \partial z_2^{n-i} }
\end{gather*}
then the map $(f,g)\mapsto \psi_n(f,g)$ is a bilinear and
$\mathrm{SL}_2\, (\mathbb{C})$-equivariant map from $V(d_1)\times V(d_2)$ onto $V(d_1 +d_2 -2n)$. The map
\begin{gather*}
V(d_1)\otimes V(d_2)\to \bigoplus_{n=0}^{\mathrm{min}(d_1 , d_2)} V(d_1+d_2-2n) \\
(f,g) \mapsto \sum_{n=0}^{\mathrm{min}(d_1 , d_2)} \psi_n(f,g) 
\end{gather*}
is an isomorphism of $\mathrm{SL}_2\, (\mathbb{C})$-modules ("Clebsch-Gordan decomposition"), cf. \cite{B-S}, p. 122. The maps $\psi_n$ are
called transvectants (\emph{\"Uberschiebungen} in German). Their importance derives from the fact that they make the preceding isomorphism
\emph{explicit}.

Now let $G:= \mathrm{SL}_3 (\CC )$, and let  $V(a, \: b)$ be the irreducible $G$-module whose highest weight has numerical labels $a, \: b$
where
$a$, $b$ are non-negative integers. A representation $V(a, \: b)\otimes V(c, \: d)$ decomposes similarly into irreducible summands, and the
Cartan-Killing theory of highest weights allows us to compute the multiplicity with which $V(e, \:f)$ occurs (an entirely similar statement
holds of course for $\mathrm{SL}_n (\CC )$ or any semi-simple linear algebraic group); in other words, the theory of highest weights asserts the
existence of an isomorphism
\[
E\otimes E' \simeq \bigoplus_{i\in I} E_i
\]
of irreducible representations of a semi-simple algebraic group, but does not give us the isomorphism, at least it is not easy to unravel from
this theory. On the other hand, it is often important to know the isomorphism, e.g.
\begin{itemize}
\item
in \emph{the problem of rationality for fields of invariants}, see \cite{Dolg1} for a survey. Here one almost always has to check certain
nondegeneracy statements for maps of the form $E\otimes E' \to E''$ (or similar maps constructed by representation theory), and for this one
has to know the maps explicitly. Often one is dependent on computer aid when studying these maps, one needs fast methods for computing them.
\item
In \emph{the geometry of syzygies} (see \cite{Wey03}). Here one wants to understand differentials of certain chain complexes constructed by
representation theoretic means, as for example by Kempf's geometric technique based on taking direct images of Koszul complexes; here
computational efficiency is again one of the desiderata.
\end{itemize}
In the first sections of this article we give a very simple method, contained in Theorem \ref{tClebschGordan}, to obtain a basis for the space
\[
\mathrm{Hom}_G (V(a, \: b)\otimes V(c, \: d), \: V(e, \: f))\, .
\]
In particular, it enables one to immediately write down matrix representatives for the occurring maps. Moreover, Theorem \ref{tClebschGordan}
gives a factorization of all such maps into certain elementary building blocks and explicit formulas for them; during the proof,
which occupies sections 2 through 4, we also set up a natural bijection between the basis maps and the expansions of Young diagrams which 
occur in the combinatorics of the Littlewood-Richardson rule.\\
One is tempted to think that something of this sort should have been discovered before, but we could not find it in the classical or modern
literature.

In any event, for us the main reason for introducing this computational scheme is that it is the one we use and found most convenient for
applications to the problem of rationality for invariant function fields; a sample of such applications is contained in section 5.

First of
all, Theorem \ref{tClebschGordan} allows one to prove rationality for many spaces $\PP (V(a, \: b))/G$ via the double bundle method
(\cite{Bo-Ka}) where $V(a, \: b)$ is a space of mixed tensors. We prove rationality of $\PP (V(4, \: 4))/G$ as an example.

For the double bundle method one uses linear fibrations over projective spaces; one may also consider linear fibrations over more
general Grassmannians; see Proposition \ref{pFibrationOverGrass}. This was already mentioned in \cite{Shep89}, but has not yet found any
application to our knowledge. One problem is that one needs to know the stable rationality for quotients of Grassmannians $\mathrm{Grass}(k, \:
V)/\Gamma $ where $V$ is a linear representation of a reductive group $\Gamma$. In Proposition \ref{pStabRatGrassmannians} we give a criterion
for stable rationality that  applies in some cases if $\Gamma$ is a group of type $\mathrm{SL}$. Using this and Theorem \ref{tClebschGordan},
we prove the rationality of the moduli space of plane curves of degree $34$, i.e. $\PP (V(0,\: 34))/G$, in Theorem \ref{tRationalityV34}. This
case cannot be handled by the double bundle method, cf. Remark \ref{rNoAlternatives}, nor has it been treated by any other method so far.

\section{The Littlewood Richardson rule for $\slC$ }

It is well known that isomorphism classes of irreducible $\mathrm{GL}_n (\CC )$-modules correspond bijectively to $n$-tuples of integers
$\lambda =(\lambda_1,
\dots , \: \lambda_n)$ with $\lambda_1 \ge \lambda_2 \ge \dots \ge \lambda_n$ via associating to such a representation its highest weight
$\lambda_1
\epsilon_1 +
\dots + \lambda_n
\epsilon_n$ where $\epsilon_i$ is the $i$-th coordinate function of the standard diagonal torus in $\mathrm{GL}_n (\CC )$. The space of the
corresponding irreducible representation will be denoted $\Sigma^{\lambda } (\CC^n )$. Here $\Sigma^{\lambda }$ is called the \emph{Schur functor}
(cf. \cite{Fu-Ha}). If all $\lambda_j$ are non-negative, one associates to $\lambda$ the corresponding \emph{Young diagram} whose
number of boxes in its $i$-th row is $\lambda_i$; $\lambda$ will often be identified with this Young diagram. For example,
\setlength{\unitlength}{1cm}
\begin{center}
\begin{picture}(4,2)
\put(-3 , 1.5){$\Sigma^{1,1,1} (\CC^3)$}
\put(-1 , 1.5){$\longleftrightarrow$ }
\put(1, 1.5){$\Lambda^3 (\CC^3)$}
\put(3, 1.5){$\longleftrightarrow$}
\thicklines
\put(4.5,1.7){\framebox(0.4,0.4){}}
\put(4.5,1.3){\framebox(0.4,0.4){}}
\put(4.5,0.9){\framebox(0.4,0.4){}}
\end{picture}
\end{center}
We list some properties of the Schur functors for future use:
\begin{itemize}
\item
One has $\Sigma^{\lambda }(\CC^n )\simeq \Sigma^{\mu} (\CC^n)$ as $\mathrm{SL}_n (\CC)$-representations if and only if $\lambda_i -\mu_i =: h$ is
constant for all $i$. In fact, in this case
\[
\Sigma^{\lambda} (\CC^n) \simeq \Sigma^{\mu }(\CC^n) \otimes \left( \Lambda^n (\CC^n)  \right)^{\otimes h}\, .
\]
\item
$\Sigma^{(\lambda_1,\: \lambda_2, \dots, \: \lambda_n )} (\CC^n)^{\vee } \simeq \Sigma^{(-\lambda_n, \: -\lambda_{n-1}, \dots, \: -\lambda_1)}
(\CC^n )$.
\item
The representation $V(a, \: b)$ of $G=\mathrm{SL}_3 (\CC )$ is isomorphic to $\Sigma^{(a+b, \: b , \: 0)} (\CC^3)$.
\item
For a Young diagram $\lambda$ with more than $n$ rows one has $\Sigma^{\lambda } (\CC^n ) = 0$ by definition.
\end{itemize}

The Littlewood-Richardson rule to decompose $\Sigma^{\lambda }\otimes \Sigma^{\mu}$ into 
irreducible factors where $\lambda$, $\mu$ are Young diagrams (cf. \cite{Fu-Ha}, \S A.1) says the 
following (in this notation we suppress the space which the Schur functors are applied to, since it plays no role): label each box of $\mu$ with
the number of the row it belongs to. Then expand the Young  diagram $\lambda$ by adding the boxes of $\mu$ to the rows of
$\lambda$ subject to the following rules:
\begin{itemize}
\item[(a)]
The boxes with labels $\le i$ of $\mu$ together with the boxes of $\lambda$ form again a Young diagram;

\item[(b)]
No column contains boxes of $\mu$ with equal labels.

\item[(c)]
When the integers in the boxes added are listed from right to left and from top down, then, for any 
$0\le s \le$ (number of boxes of $\mu$), the first $s$ entries of the list satisfy: each label $l$ 
($1\le l\le$ (number of rows of $\mu$)$-1$ ) occurs at least as many times as the label $l+1$.

\end{itemize}
We will call this configuration of boxes (together with the labels) a $\mu$-\emph{expansion of} $\lambda$. 
Then the multiplicity of $\Sigma^{\nu}$ in $\Sigma^{\lambda}\otimes\Sigma^{\mu}$ is the number 
of times the Young diagram $\nu$ can be obtained by expanding $\lambda$ by $\mu$ according to 
the above rules, forgetting the labels.

\begin{example}
For $\Sigma^{(2,\: 1, \: 0)}\otimes \Sigma^{(2,\: 1, \: 0)}$ the following expansions are possible:
\setlength{\unitlength}{1cm}
\begin{center}
\begin{picture}(4,4.5)
\thicklines
\put(-4,4){\framebox(0.4,0.4){}}
\put(-3.6,4){\framebox(0.4,0.4){}}
\thinlines
\put(-3.2,4){\framebox(0.4,0.4){1}}
\put(-2.8,4){\framebox(0.4,0.4){1}}
\thicklines
\put(-4,3.6){\framebox(0.4,0.4){}}
\thinlines
\put(-3.6,3.6){\framebox(0.4,0.4){2}}

\thicklines
\put(-1,4){\framebox(0.4,0.4){}}
\put(-0.6,4){\framebox(0.4,0.4){}}
\thinlines
\put(-0.2,4){\framebox(0.4,0.4){1}}
\put(0.2,4){\framebox(0.4,0.4){1}}
\thicklines
\put(-1,3.6){\framebox(0.4,0.4){}}
\thinlines
\put(-1,3.2){\framebox(0.4,0.4){2}}

\thicklines
\put(2,4){\framebox(0.4,0.4){}}
\put(2.4,4){\framebox(0.4,0.4){}}
\thinlines
\put(2.8,4){\framebox(0.4,0.4){1}}
\thicklines
\put(2,3.6){\framebox(0.4,0.4){}}
\thinlines
\put(2.4,3.6){\framebox(0.4,0.4){1}}
\put(2.8,3.6){\framebox(0.4,0.4){2}}

\thicklines
\put(4.5,4){\framebox(0.4,0.4){}}
\put(4.9,4){\framebox(0.4,0.4){}}
\thinlines
\put(5.3,4){\framebox(0.4,0.4){1}}
\thicklines
\put(4.5,3.6){\framebox(0.4,0.4){}}
\thinlines
\put(4.9,3.6){\framebox(0.4,0.4){1}}
\put(4.5,3.2){\framebox(0.4,0.4){2}}

\thicklines
\put(-3,2){\framebox(0.4,0.4){}}
\put(-2.6,2){\framebox(0.4,0.4){}}
\thinlines
\put(-2.2,2){\framebox(0.4,0.4){1}}
\put(-3,1.2){\framebox(0.4,0.4){1}}
\thicklines
\put(-3,1.6){\framebox(0.4,0.4){}}
\thinlines
\put(-2.6,1.6){\framebox(0.4,0.4){2}}

\thicklines
\put(0,2){\framebox(0.4,0.4){}}
\put(0.4,2){\framebox(0.4,0.4){}}
\thinlines
\put(0.8,2){\framebox(0.4,0.4){1}}
\put(0,0.8){\framebox(0.4,0.4){2}}
\thicklines
\put(0,1.6){\framebox(0.4,0.4){}}
\thinlines
\put(0,1.2){\framebox(0.4,0.4){1}}

\thicklines
\put(3,2){\framebox(0.4,0.4){}}
\put(3.4,2){\framebox(0.4,0.4){}}
\thinlines
\put(3,1.2){\framebox(0.4,0.4){1}}
\thicklines
\put(3,1.6){\framebox(0.4,0.4){}}
\thinlines
\put(3.4,1.6){\framebox(0.4,0.4){1}}
\put(3.4,1.2){\framebox(0.4,0.4){2}}

\thicklines
\put(5.5,2){\framebox(0.4,0.4){}}
\put(5.9,2){\framebox(0.4,0.4){}}
\thinlines
\put(5.5,0.8){\framebox(0.4,0.4){2}}
\thicklines
\put(5.5,1.6){\framebox(0.4,0.4){}}
\thinlines
\put(5.9,1.6){\framebox(0.4,0.4){1}}
\put(5.5,1.2){\framebox(0.4,0.4){1}}

\end{picture}\end{center}
Hence we have the following decomposition
\[
V(1, \: 1) \otimes V(1, \: 1) = V(2, \: 2) \oplus V(3, \: 0) \oplus V(0, \: 3) \oplus 2 \, V(1,\: 1) \oplus V(0, \: 0)\, .
\]
\end{example}

\begin{figure}
\includegraphics*[width=8cm]{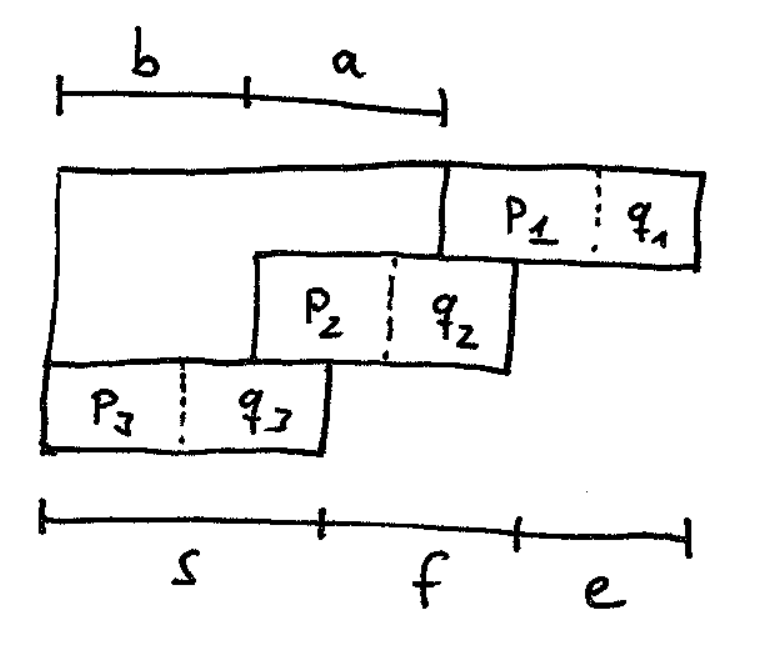}
\caption{}
\label{fpq}
\end{figure}

For $G=\slC$ the combinatorics of the Littlewood-Richardson rule can be handled explicitly. For this let
$V(e,f)$ be a summand of $V(a,b) \otimes V(c,d)$. In the following we set $\lambda = (a+b,b,0)$, $\mu = (c+d,d,0)$ and let $\nu = (e+f+s,f+s,s)$ be the unique Young diagram corresponding to $V(e,f)$ in the
decomposition of $\Sigma^\lambda (\CC^3) \otimes \Sigma^\mu (\CC^3)$.

\begin{lemma} \label{lpq}
Expand the Young diagram of $\lambda$ by adding $p_i$ boxes with label $1$ to row $i$ and afterwards
$q_i$ boxes with label $2$ to row $i$ (see Figure \ref{fpq}). This is a $\mu$-expansion of $\lambda$ if and only if the following inequalities hold:
\begin{enumerate}
	\item \begin{enumerate}
		\item $p_i \ge 0$ \xlabel{iPositivityP}
		\item $q_i \ge 0$ \xlabel{iPositivityQ}
		\end{enumerate}
	\item \begin{enumerate}
		\item $p_2 \le a$ \xlabel{iOverlap1row2}
		\item $p_3 \le b$ \xlabel{iOverlap1row3}			
		\item $p_2+q_2-a\le p_1$ \xlabel{iOverlap2row2}
		\item $p_3+q_3-b\le p_2$ \xlabel{iOverlap2row3}
		\end{enumerate}
	\item \begin{enumerate}		
		\item $q_1 = 0$ \xlabel{iString1}
		\item $q_2\le p_1$ \xlabel{iString2}
		\item $q_2+q_3\le p_1+p_2$ \xlabel{iString3}
		\end{enumerate}
	\item \begin{enumerate}
		\item $p_1+p_2+p_3 = c+d$ \xlabel{iTotal1}
		\item $q_1+q_2+q_3 = d$ \xlabel{iTotal2}
		\end{enumerate}
\end{enumerate}	
\end{lemma}

\begin{proof}
The inequalities (\ref{iPositivityP}) and (\ref{iPositivityQ}) are obvious positivity conditions. (\ref{iOverlap1row2}) and (\ref{iOverlap1row3}) ensure that the boxes of $\lambda$ together with the boxes of $\mu$ with label $1$ form again a Young diagram and there is at most one label $1$ in every column.  (\ref{iOverlap2row2}) and (\ref{iOverlap2row3}) guarantee that the boxes of $\lambda$ together with all boxes of $\mu$ form again a Young diagram and there is at most one label $2$ in every column.  (\ref{iString1}), (\ref{iString2}) and (\ref{iString3}) encode that the string of
labels read from right to left and from top down always contains more $1$'s then $2$'s. The last two equations reflect that the total number of $1$'s and $2$'s is given by the Young diagram describing $V(c,d)$.
\end{proof}

For given $a,b,c,d,e,f$ the equations above leave only one unknown:

\begin{lemma} \label{lOneUnknown}
Let $s=p_3+q_3$ be the number of labelled boxes in the third row of the $\mu$-expansion. Let furthermore $j = p_3$ be the number of $1$'s in the third row 
and $t = p_2-q_3$ the difference between the number of $1$'s in the second row and the number of $2$'s in the third row. With this we obtain
\begin{enumerate}
	\item $p_3 = j$ \label{ip3}
	\item $q_3 = s-j$ \label{iq3}
	\item $p_2 = s+t-j$ \label{ip2}
	\item $q_2 = d-s+j$ \label{iq2}
	\item $p_1 = c+d-(s+t)$ \label{ip1}
	\item $q_1= 0$ \label{iq1}
	\item $s = \frac{(a+c-e)+2(b+d-f)}{3}$ \label{is}
	\item $t = \frac{(a+c-e)-(b+d-f)}{3}$ \label{it}
\end{enumerate}
\end{lemma}

\begin{proof}
(\ref{ip3}), (\ref{iq3}) and (\ref{ip2}) follow from the definition of $j$, $s$ and $t$. Since the total number
of $2$'s is $d$ we obtain (\ref{iq2}). Similarly $p_1+p_2+p_3 = c+d$ implies (\ref{ip1}). Equation (\ref{iq1}) is true for all $\mu$-expansions. Since we know that the number of labelled boxes  is $c+2d$, the number of empty boxes is $a+2b$ and the total number of boxes is $3s+2f+e$, we obtain
\[
	s = \frac{(a+c-e)+2(b+d-f)}{3}.
\]
Finally the total length of the first row is $a+b+p_1$, on the one hand, and $s+e+f$ on the other. This gives (\ref{it}).
\end{proof}

\begin{proposition}
For given $a$, $b$, $c$, $d$, $e$ and $f$ there exists a $\mu$-expansion of $\lambda$ of shape $\nu$ with $p_3=j$ if and only if $j$ satisfies the following inequalities:
\begin{enumerate}
	\item \begin{enumerate}	
		\item $0 \le j \le s+t \le c+d$
		\item $0 \le s-j \le d$
		\end{enumerate}
	\item \begin{enumerate}	
		\item[(a)] $s+t-j \le a$
		\item[(b)] $j \le b$
		\item[(d)] $j \le b+t$ 
		\end{enumerate}
	\item \begin{enumerate}	
		\item[(b)] $j \le c-t$
		\item[(c)] $j \le c$
		\end{enumerate}
\end{enumerate}
\end{proposition}

\begin{proof}
Substitute the expressions of Lemma \ref{lOneUnknown} into the inequalities of Lemma \ref{lpq}. The inequality (\ref{iOverlap2row2}) gives $s+2t \le a+c$ which is always true since $s+2t = a+c-e$. Furthermore (\ref{iString1}), (\ref{iTotal1}) and (\ref{iTotal2}) simplify to $0 = 0$.
\end{proof}

\begin{remark}
The numbering in the list above is taken from the corresponding inequalities in 
Lemma \ref{lpq}.
\end{remark}

\section{A basis for $\mathrm{Hom}_G \bigl(V(a, \: b)\otimes \: V(c, \: d), \: V(e, \: f)\bigr)$}
We put

\begin{gather*}
S^a := \mathrm{Sym}^a (\CC^3), \quad D^b := \mathrm{Sym}^b (\CC^3)^{\vee}
\end{gather*}

and denote by $e_1,\: e_2, \: e_3$ and $x_1, \: x_2, \: x_3$ dual bases in $\CC^3$ resp.
$(\CC^3)^{\vee}$ so that $V(a,\: b)$ can be realized concretely as the kernel of the map

\begin{gather}\label{formulaDelta}
\Delta := \sum_{i=1}^3 \frac{\partial}{\partial e_i}\otimes \frac{\partial}{\partial x_i} \, :\, S^a
\otimes D^b \to S^{a-1} \otimes D^{b-1}\, ;
\end{gather}

We will always view $V(a,\: b)$ in this way in the following. 
By $\pi_{e,\: f}$ we denote the equivariant projection from $S^e \otimes D^f$ onto $V(e, \: f)$. 

Our purpose is to determine an explicit basis of the
$G$-equivariant maps
\[
\mathrm{Hom}_G \bigl(V(a, \: b)\otimes \: V(c, \: d), \: V(e, \: f)\bigr)
\]
if $V(e, \: f)$ is a subrepresentation of $V(a, \: b) \otimes V(c, \: d)$ To this end we define the following elementary
maps:

\begin{gather*}
\alpha \, : \, (S^a \otimes D^b) \otimes (S^c \otimes D^d) \to (S^{a-1} \otimes D^b) \otimes (S^c \otimes D^{d-1})\\
\alpha := \sum_{i=1}^{3} \frac{\partial}{\partial e_i} \otimes \mathrm{id} \otimes \mathrm{id} \otimes \frac{\partial
}{\partial x_i}\nonumber
\end{gather*} 

\begin{gather*}
\beta \, : \, (S^a \otimes D^b) \otimes (S^c \otimes D^d) \to (S^{a} \otimes D^{b-1}) \otimes (S^{c-1} \otimes D^{d})\\
\beta := \sum_{i=1}^{3} \mathrm{id} \otimes \frac{\partial}{\partial x_i}  \otimes
\frac{\partial }{\partial e_i} \otimes \mathrm{id} \nonumber
\end{gather*} 

\begin{gather*}
\vartheta \, : \, (S^a \otimes S^c) \otimes (D^{b+d}) \to (S^{a-1} \otimes S^{c-1}) \otimes D^{b+d+1}\\
\vartheta := \sum_{\sigma\in\mathfrak{S}_3} (-1)^{\mathrm{sgn}(\sigma )}\frac{\partial}{\partial e_{\sigma (1)}}\otimes \frac{\partial }{\partial
e_{\sigma (2)}}
\otimes x_{\sigma (3)} 
\nonumber
\end{gather*}

\begin{gather*} 
\omega \, : \, S^{a+c} \otimes (D^{b}\otimes D^d) \to S^{a+c+1} \otimes (D^{b-1} \otimes D^{d-1})\\
\omega := \sum_{\sigma \in \mathfrak{S}_3} (-1)^{\mathrm{sgn}(\sigma )} e_{\sigma (1)} \otimes \frac{\partial}{\partial x_{\sigma (2)}}\otimes
\frac{\partial }{\partial x_{\sigma (3)}}  
\nonumber
\end{gather*}

Note that an easier way of defining $\vartheta$ and $\omega$ is by saying that $\vartheta$ is multiplication by the
determinant $x_1\wedge x_2 \wedge x_3$ and $\omega$ multiplication by its inverse $e_1 \wedge e_2 \wedge e_3$.

\begin{theorem}\xlabel{tClebschGordan}
Suppose that $V(e, \: f)$ occurs in the decomposition of $V(a,\: b)\otimes V(c,\: d)$ and let $s$ and
$t$ be defined as above. Let $J$ be the set of all integers $j$ satisfying the inequalities
\begin{enumerate}
	\item \begin{enumerate}	
		\item $0 \le j \le s+t \le c+d$
		\item $0 \le s-j \le d$
		\end{enumerate}
	\item \begin{enumerate}	
		\item[(a)] $s+t-j \le a$
		\item[(b)] $j \le b$
		\item[(d)] $j \le b+t$ 
		\end{enumerate}
	\item \begin{enumerate}	
		\item[(b)] $j \le c-t$
		\item[(c)] $j \le c$
		\end{enumerate}
\end{enumerate}
Then a basis of $\mathrm{Hom}_G (V(a,\: b) \otimes V(c, \: d), \: V(e, \: f))$ is given by the restriction to 
$V(a,\: b) \otimes V(c,\:d)$ of the maps
\[
	\pi_{e,\: f}\circ \vartheta^t \circ \beta^{j} \circ \alpha^{s-j}, \quad j\in J
\]
if $t\ge0$ and
\[
	\pi_{e,\: f}\circ \omega^{-t} \circ \beta^{j} \circ \alpha^{s+t-j}, \quad j \in J
\]
if $t \le 0$.
\end{theorem}

A few explanatory remarks are in order.

\begin{remark}\xlabel{rTacitMultiplication}
When writing a composition like $\pi_{e,\:f}\circ \vartheta^t \circ \beta^{j} \circ \alpha^{s-j}$, we suppress the
obvious multiplication maps from the notation. For example if $t\ge0$ the map
\[
\beta^{j} \circ \alpha^{s-j}\, :\, (S^a\otimes D^b) \otimes (S^{c}\otimes D^{d} ) \to (S^{a-s+j}\otimes D^{b-j})
\otimes (S^{c-j} \otimes D^{d-s+j} )
\]
is composed with the multiplication map
\[
(S^{a-s+j}\otimes D^{b-j})
\otimes (S^{c-j} \otimes D^{d-s+j} ) \to S^{a-s+j} \otimes S^{c-j} \otimes D^{b+d-s}
\]
before applying $\vartheta^t$ to land in $S^{a-s+j-t} \otimes S^{c-j-t} \otimes D^{b+d-s+t}$. Before applying the equivariant
projection $\pi_{e,\: f}$ we multiply again to map to
\[
S^{a+c-s-2t}\otimes D^{b+d-s+t} 
\]
which one, looking back at the definition of $t$ and $s$, identifies as $S^e\otimes D^f$. This simplification of notation
should cause no confusion.
\end{remark} 

\begin{proof}
Note that the element $m := (e_1^a\otimes x_3^b) \otimes (e_3^c\otimes x_1^d)$ is in the subspace $V(a, \: b ) \otimes V(c,
\: d)
\subset (S^a \otimes D^b) \otimes (S^c \otimes D^d)$ by the definition of $\Delta$ in formula (\ref{formulaDelta}). Note also that the image of 
the
map
\begin{gather}\label{formuladelta}
\delta \, :\, S^{e-1} \otimes D^{f-1} \to S^e \otimes D^f, \quad \delta = \sum_{i=1}^{3} e_i \otimes x_i
\end{gather}
is a complement to the subspace $V(e, \: f)$ in $S^e \otimes D^f$. 
If $t\ge0$ we compute
\[
(\vartheta^t \circ \beta^{j}\circ \alpha^{s-j}) (m) = (\mathrm{nonzero}\;\mathrm{constant}) \cdot e_1^{a-s+j-t}e_3^{c-j-t}\otimes x_3^{b-j}x_2^t
x_1^{d-s+j}.
\]
The inequalities above imply that this is a non-zero monomial in $S^e\otimes D^f$ for all $j\in J$.
If $t\le 0$ then 
\[
(\omega^{-t}\circ \beta^{j}\circ \alpha^{s+t-j}) (m) = (\mathrm{nonzero}\;\mathrm{constant}) \cdot e_1^{a-s-t+j} e_2^{-t} e_3^{c-j} \otimes
x_3^{b+t-j} x_1^{d-s+j}\, 
\]
is also a non-zero monomial in $S^e \otimes D^f$. 
Each nonzero bihomogeneous polynomial in the subspace \[ \mathrm{im} (\delta ) = (e_1\otimes x_1 +
e_2\otimes x_2 + e_3\otimes x_3)\cdot (S^{e-1}\otimes D^{f-1}) \subset S^e \otimes D^f \] contains monomials (with nonzero coefficient) divisible
by
$e_2\otimes x_2$. Since the preceding monomials in cases $t\ge 0$ resp. $t\le 0$ are not divisible by $e_2\otimes x_2$, a linear combination of
them can be zero modulo $\mathrm{im} (\delta )$ only if this linear combination is already zero as a polynomial in $S^e\otimes D^f$. But in both
cases $t\ge 0$ and $t\le 0$, the degrees of the above monomials with respect to the variable $e_1$ are pairwise distinct, so they cannot combine
to zero nontrivially in
$S^e\otimes D^f$.
\end{proof}

\section{Equivariant projections}

To complete the picture, we will give in this section a method to compute the equivariant projection
\[
\pi_{a, \: b} \, :\, S^a \otimes D^b \to V(a, \: b) = \mathrm{ker}(\Delta )\subset S^a \otimes D^b \, .
\]

\begin{lemma}\xlabel{lPolynomialNature}
One has
\[
\pi_{a, \: b} = \sum_{j=0}^N \mu_j \delta^j \Delta^j
\]
for some $N\in\mathbb{N}$ and certain $\mu_j\in\QQ$ (the map $\delta$ is defined in formula \ref{formuladelta}). 
\end{lemma}

\begin{proof}
Let us denote by $\pi_{a,\: b, \: i}$ the equivariant projection
\[
\pi_{a, \: b, \: i}\, :\, S^a \otimes D^b \to V(a-i, \: b-i) \subset S^a \otimes D^b
\]
so that $\pi_{a, \: b} = \pi_{a, \: b, \: 0}$. Look at the diagram
\xycenter{
S^a\otimes D^b \ar[r]^{\Delta^i} &S^{a-i}\otimes D^{b-i} \ar[d]^{\pi_{a-i,\: b-i}}\\
&V(a-i, \: b-i)\subset S^{a-i} \otimes D^{b-i} \ar[lu]^{\delta^i}
	}
By Schur's lemma,
\begin{gather}\label{formulaLambdai}
\pi_{a, \: b, \: i} = \lambda_i \delta^i \pi_{a-i, \: b-i} \Delta^i
\end{gather}
for some nonzero constants $\lambda_i$. On the other hand,
\[
\pi_{a, \: b} = \mathrm{id} - \sum_{i=1}^{\mathrm{min}(a, \: b)} \pi_{a, \: b, \: i} \, .
\]
Therefore, since the assertion of the Lemma holds trivially if one of $a$ or $b$ is zero, the general case follows by induction on $\min(a,b)$.
\end{proof}

Note that to compute the $\mu_j$ in the expression of $\pi_{a, \: b}$ in Lemma \ref{lPolynomialNature}, it suffices to calculate the $\lambda_i$ in
formula \ref{formulaLambdai} which can be done by the rule
\[
\frac{1}{\lambda_i} (e_1^{a-i} \otimes x_3^{b-i}) = \left( 
\Delta^i \circ \delta^{i}\right) (e_1^{a-i} 
\otimes x_3^{b-i})\, 
\]
which uses (\ref{formulaLambdai}) and the injectivitiy of $\delta^i$.

\section{Examples and applications}

In the following example we write down explicit matrix representatives for the maps given in \ref{tClebschGordan} in one special case.

\begin{example}\xlabel{eMatricesV(1,1)}
In the decomposition of $V(1,\: 1) \otimes V(1, \: 1)$, the representation $V(1,\: 1)$ occurs with multiplicity $2$, corresponding to a two
dimensional space
\[
V(1, \: 1) \otimes V(1,\: 1) \to V(1, \: 1)
\]
of $\mathrm{SL}_3 (\CC )$-equivariant maps. Here $a=b=c=d=e=f=s=1$ and $t=0$. Therefore a basis
for this space of equivariant homomorphisms is given by $\alpha$ and $\beta$. 

To give matrix representatives of $\alpha$ and $\beta$ we use  the vectors
\begin{gather*}
{q}_{12}={e}_{1} {x}_{2}, \,
{q}_{13}={e}_{1} {x}_{3} \\
{q}_{21}={e}_{2} {x}_{1}, \,
{q}_{23}={e}_{2} {x}_{3}\\
{q}_{31}={e}_{3} {x}_{1}, \,
{q}_{32}={e}_{3} {x}_{2}\\
{q}_{22}={e}_{1} {x}_{1}-{e}_{2} {x}_{2}, \,
{q}_{33}={e}_{1} {x}_{1}-{e}_{3} {x}_{3}\\
\end{gather*}
(in this order) as a basis of the $8$-dimensional space $V(1,\: 1)$. Using the definition of $\alpha$ and $\beta$ we obtain:
\begin{align*}
\alpha &= \left(\begin{smallmatrix}0&
      0&
      -\frac{2}{3} {q}_{22}+\frac{1}{3} {q}_{33}&
      0&
      {q}_{32}&
      0&
      {q}_{12}&
      {q}_{12}\\
      0&
      0&
      {q}_{23}&
      0&
      \frac{1}{3} {q}_{22}-\frac{2}{3} {q}_{33}&
      0&
      {q}_{13}&
      {q}_{13}\\
      \frac{1}{3} {q}_{22}+\frac{1}{3} {q}_{33}&
      0&
      0&
      0&
      0&
      {q}_{31}&
      {-{q}_{21}}&
      0\\
      {q}_{13}&
      0&
      0&
      0&
      0&
      \frac{1}{3} {q}_{22}-\frac{2}{3} {q}_{33}&
      {-{q}_{23}}&
      0\\
      0&
      \frac{1}{3} {q}_{22}+\frac{1}{3} {q}_{33}&
      0&
      {q}_{21}&
      0&
      0&
      0&
      {-{q}_{31}}\\
      0&
      {q}_{12}&
      0&
      -\frac{2}{3} {q}_{22}+\frac{1}{3} {q}_{33}&
      0&
      0&
      0&
      {-{q}_{32}}\\
      {-{q}_{12}}&
      0&
      {q}_{21}&
      0&
      {q}_{31}&
      {-{q}_{32}}&
      -\frac{1}{3} {q}_{22}+\frac{2}{3} {q}_{33}&
      \frac{1}{3} {q}_{22}+\frac{1}{3} {q}_{33}\\
      0&
      {-{q}_{13}}&
      {q}_{21}&
      {-{q}_{23}}&
      {q}_{31}&
      0&
      \frac{1}{3} {q}_{22}+\frac{1}{3} {q}_{33}&
      \frac{2}{3} {q}_{22}-\frac{1}{3} {q}_{33}\\
      \end{smallmatrix}\right)\\
\beta &= \left(\begin{smallmatrix}0&
      0&
      \frac{1}{3} {q}_{22}+\frac{1}{3} {q}_{33}&
      {q}_{13}&
      0&
      0&
      {-{q}_{12}}&
      0\\
      0&
      0&
      0&
      0&
      \frac{1}{3} {q}_{22}+\frac{1}{3} {q}_{33}&
      {q}_{12}&
      0&
      {-{q}_{13}}\\
      -\frac{2}{3} {q}_{22}+\frac{1}{3} {q}_{33}&
      {q}_{23}&
      0&
      0&
      0&
      0&
      {q}_{21}&
      {q}_{21}\\
      0&
      0&
      0&
      0&
      {q}_{21}&
      -\frac{2}{3} {q}_{22}+\frac{1}{3} {q}_{33}&
      0&
      {-{q}_{23}}\\
      {q}_{32}&
      \frac{1}{3} {q}_{22}-\frac{2}{3} {q}_{33}&
      0&
      0&
      0&
      0&
      {q}_{31}&
      {q}_{31}\\
      0&
      0&
      {q}_{31}&
      \frac{1}{3} {q}_{22}-\frac{2}{3} {q}_{33}&
      0&
      0&
      {-{q}_{32}}&
      0\\
      {q}_{12}&
      {q}_{13}&
      {-{q}_{21}}&
      {-{q}_{23}}&
      0&
      0&
      -\frac{1}{3} {q}_{22}+\frac{2}{3} {q}_{33}&
      \frac{1}{3} {q}_{22}+\frac{1}{3} {q}_{33}\\
      {q}_{12}&
      {q}_{13}&
      0&
      0&
      {-{q}_{31}}&
      {-{q}_{32}}&
      \frac{1}{3} {q}_{22}+\frac{1}{3} {q}_{33}&
      \frac{2}{3} {q}_{22}-\frac{1}{3} {q}_{33}\\
      \end{smallmatrix}\right)
\end{align*}

Notice that $\alpha = \beta^t$.
\end{example}

Theorem \ref{tClebschGordan} is of particular importance in applications to the question of rationality of quotient spaces $\PP (V(a,\: b))/G$.
In the following, if a linear algebraic group $\Gamma$ acts on a variety $X$, the quotient $X/\Gamma$ is always taken in the sense of
Rosenlicht: there is a non-empty $\Gamma$-invariant open subset $U\subset X$ for which a geometric quotient $U/\Gamma$ exists, and $X/\Gamma$
denotes any birational model for this quotient.\\
We need the following extension of the double bundle method of \cite{Bo-Ka}, see also \cite{Shep89}.

\begin{proposition}\xlabel{pFibrationOverGrass}
Let $V$ and $W$ be representations of a connected reductive group $\Gamma$ with $\dim V - \dim W =:k > 0$. Let $U$ be a subrepresentation of
$\mathrm{Hom}(V, \: W)$ such that for generic $u\in U$ the corresponding map in $\mathrm{Hom}(V, \: W)$ has full rank so that we get a rational
map
\begin{align*}
\varphi\, :\, \PP (U) &\dasharrow \mathbb{G}:=\mathrm{Grass}(k,\: V)\\
  [u]&\mapsto \mathrm{ker}(u) \, .
\end{align*}
Let us assume furthermore that
\begin{itemize}
\item
$\varphi$ is dominant; this is equivalent to saying that a fibre $\varphi^{-1}(\varphi ([u]))$ has dimension $\dim \PP (U) - \dim \mathbb{G}$.
\item
$\mathbb{G}/\Gamma$ is stably rational in the sense that $(\mathbb{G}/\Gamma)\times\PP^r$ is rational for some $r\le
\dim \PP (U ) - \dim \mathbb{G}$.
\item
Let $Z$ be the kernel of the action of $\Gamma$ on $\mathbb{G}$: we require the existence of a $\Gamma /Z$-linearized very ample
line bundle
$\mathcal{M}$ on
$\mathbb{G}$ such that for the embedding $\mathbb{G}$ in $\PP (H^0 (\mathcal{M})^{\vee})$ the locus of very stable
points in
$\mathbb{G}$ (i.e. stable with trivial stabilizer in $\Gamma/Z$) is nonempty, $Z$ acts trivially on $\PP (U)$, and there exists a $\Gamma /Z$-linearized
line bundle
$\mathcal{L}$ on the product $\PP (U) \times \mathbb{G}$ cutting out $\mathcal{O}(1)$ on the fibres of the projection to $\mathbb{G}$.
\end{itemize}
Then $\PP(U)/\Gamma$ is rational.
\end{proposition}

\begin{proof}
Let $X:=$the (closure of) the graph of $\varphi$, $p\, :\, X\to \mathbb{G}$ the restriction of the projection which (maybe after shrinking
$\mathbb{G}$) we may assume to be a projective space bundle for which $\mathcal{L}$ is a relatively ample bundle cutting out
$\mathcal{O}(1)$ on the fibres. The main technical point is the following result from descent theory (\cite{Mum},
\S 7.1,
\cite{Shep}, Thm. 1): for all sufficiently large $n$, putting $\mathbb{G}_0:=$ the locus of very stable points in $\mathbb{G}$ with respect
to
$\mathcal{M}$,
$X_0:=$the locus of very stable points in $X$ w.r.t. $\mathcal{L}\otimes p^{\ast}(\mathcal{M})^{\otimes n}$, one has
$p^{-1}(\mathbb{G}_0)\subset X_0$, and a cartesian diagram
\[
\begin{CD}
X_0 @>>> X_0 / (\Gamma /Z)\\
@V{p}VV        @V{\bar{p}}VV\\
\mathbb{G}_0 @>>> \mathbb{G}_0/(\Gamma /Z)
\end{CD}
\]
such that $\mathcal{L}$ descends to a line bundle $\bar{\mathcal{L}}$ on $X_0/(\Gamma /Z)$ cutting out $\mathcal{O}(1)$ on the fibres of
$\bar{p}$. Hence $\bar{p}$ is also a Zariski locally trivial projective bundle (of the same rank as $p$). 
It then follows that $\PP(U)/\Gamma$ is rational.
\end{proof}

Theorem \ref{tClebschGordan} in conjunction with Proposition \ref{pFibrationOverGrass} yields rationality results for spaces of mixed tensors
of which the following is a sample:

\begin{theorem}\xlabel{tRatV44}
The space $\PP (V(4,\: 4))/\mathrm{SL}_3 (\CC )$ is rational.
\end{theorem}

\begin{proof}
In fact
\begin{gather*}
V(1, \: 7) \subset V(4,\: 4)\otimes V(2, \: 5)\, , \\
\dim V(4, \: 4)= 125, \: \dim V(2, \: 5)= 81, \: \dim V(1, \: 7)=80\, ,
\end{gather*}
and the multiplicity of $V(1, \: 7)$ in $V(4,\: 4)\otimes V(2, \: 5)$ is $2$. More precisely
here $s=3$ and $t=1$. The most restrictive inequality of Proposition \ref{tClebschGordan} is
\[
	0 \le j \le c-t = 1
\]
in this situation. Therefore $\psi = \vartheta \circ \beta \circ \alpha^2$ and $\phi = \vartheta \circ \alpha^3$ are
independent equivariant projections to $V(1, \; 7)$. We will use $\psi$ in this argument.

We now consider the induced map
\[
\Psi\, :\, \PP (V(4,\: 4)) \dasharrow \PP (V(2, \: 5))\, .
\]
There are stable vectors in $\PP (V(2, \: 5))$ and on $\PP (V(4, \: 4))\times \PP (V(2,\: 5))$ we can use $\mathcal{L}=\mathcal{O}(1)\boxtimes
\mathcal{O}(1)$ as $\mathrm{PGL}_3 (\CC )$-linearized line bundle. Moreover, $\PP (V(2, \: 5))$ is stably rational of level $18$ since the
action of $\mathrm{PGL}_3 (\CC )$ on pairs of $3\times 3$ matrices by simultaneous conjugation is almost free, and the quotient is known to be
rational. 

Now consider a point $x_0\in V(4, \: 4)$. If the map 
$$\psi (x_0 , \cdot )\, : V(2,\: 5) \to V(1, \: 7)$$ 
has maximal rank $80$, $\Psi$ is well defined. In this situation let  $y_0$ be a generator of $\ker \psi(x_0 , \cdot)$. If the map 
$$\psi (\cdot , \: y_0) \, :\, V(4,\: 4)\to V(1, \: 7)$$
has also rank $80$ we obtain that the fibre $\Psi^{-1}(\Psi ([x_0]))$ has the expected dimension.
For a random $x_0$ it is straightforward to check all of this using a computer algebra program. 
Notice that this can even be checked over a finite field, since the rank of a matrix is semicontinuous over $\rmspec \ZZ$. 
See \cite{script} for a Macaulay2-script. We can therefore apply Proposition \ref{pFibrationOverGrass}
and obtain that $\PP (V(4,\: 4))/\mathrm{SL}_3 (\CC )$ is rational.
\end{proof}

The following result allows us to make use of Grassmannians other than projective spaces in some cases as well. 

\begin{proposition}\xlabel{pStabRatGrassmannians}
Let $V$ be a (finite dimensional as always) representation of $\Gamma = \mathrm{SL}_p (\CC )$, $p$ prime. Let $\mathbb{G}:=\mathrm{Grass} (k,
\: V)$ be the Grassmannian of
$k$-dimensional subspaces of
$V$. Assume:
\begin{itemize}
\item
The kernel $Z$ of the action of $\Gamma$ on $\PP (V)$ coincides with the center $\ZZ /p\ZZ$ of $\mathrm{SL}_p (\CC )$ and the action of
$\Gamma/Z$ on
$\PP (V)$ is almost free. Furthermore, the action of
$\Gamma$ on
$V$ is almost free and each element of $Z$ not equal to the identity acts homothetically as multiplication by a primitive $p$th root of unity.
\item
$k \le \dim V - \dim \Gamma-1$.
\item
$p$ does not divide $k$.
\end{itemize}
Then $\mathbb{G}/\Gamma$ is stably rational, in fact, $\mathbb{G}/\Gamma \times \PP^{\dim \Gamma +1}$ is rational.
\end{proposition}

\begin{proof}
Let $C  \subset \Lambda^k (V)$ be the affine cone over $\mathbb{G}$ consisting of pure (complety decomposable) $k$-vectors. We will show that
under the assumptions of the proposition, the action of $\Gamma$ on $C$ is almost free. This will accomplish the proof since $C/\Gamma$ is
generically a torus bundle over $\mathbb{G}/\Gamma$ hence Zariski-locally trivial since tori are special groups, and the group
$\Gamma=\mathrm{SL}_p (\CC )$ is also special. Recall that a linear algebraic group is called special if every \'{e}tale locally trivial
principal bundle for the group in question is Zariski locally trivial. See \cite{Se58} for the related theory. So $C/\Gamma \times \Gamma$ is
birational to $C$, hence rational, and $\Gamma$ is of course rational as a variety.\\ 
Let $v_1\wedge v_2 \wedge \dots \wedge v_k$ be a general $k$-vector in $\Lambda^k (V)$. Since $k \le \dim V - \dim \Gamma-1$ and, in $\PP (V)$,
$\dim (\Gamma\cdot [v_1]) = \dim \Gamma$ since $Z$ is finite and $\Gamma/Z$ acts almost freely on $\PP (V)$, the $k-1$-dimensional projective
linear subspace spanned by $v_1, \dots , v_k$ in $\PP (V)$ will intersect the $\dim V -1 - \dim \Gamma$ codimensional orbit $\Gamma\cdot [v_1]$ only
in
$[v_1]$. Hence, if an element
$g\in \Gamma$ stabilizes $v_1\wedge \dots \wedge v_k$, it must lie in $Z$. Thus $g \cdot (v_1 \wedge \dots \wedge v_k) = \zeta^{k} (v_1 \wedge
\dots
\wedge v_k)$ for a primitive $p$-th root of unity $\zeta$ if $g\neq 1$. But since $p$ does not divide $k$, the case $g\neq 1$ cannot occur.
\end{proof}

As an application we prove the following result which has not been obtained by other techniques so far.

\begin{theorem}\xlabel{tRationalityV34}
The moduli space $\PP (\mathrm{Sym}^{34} (\CC^3)^{\vee }) /\mathrm{SL}_3 (\CC )$ of plane curves of degree $34$ is rational.
\end{theorem}

\begin{proof}
We have
\[
V(0, \: 34) \subset \mathrm{Hom} (V(14,\: 1), \: V(0,\: 21))\, ,
\]
with multiplicity one and $\dim V(0, \: 34) = 630, \dim V(14,\: 1) = 255$, $\dim V(0,\: 21)=253$. In this case $s=14$ and $t=0$
and the strongest restriction in Theorem \ref{tClebschGordan} is
\[
	14 = s+t-a \le j \le b+t = 14
\]
The projection 
\[
	\psi \colon V(0,34) \otimes V(14,1) \to V(0,21)
\]
is therefore given by $\psi = \beta^{14}$. Using this we get an induced rational map
\[
\Psi \, :\, \PP (V(0,\: 34)) \dasharrow \mathrm{Grass}(2,\: V(14,\: 1))\, 
\]
with $\dim \PP (V(0,\: 34))= 629$ and $\dim \mathrm{Grass}(2,\: V(14,\: 1)) = 506$. Moreover, Proposition \ref{pStabRatGrassmannians} shows that
$\mathrm{Grass}(2,\: V(14,\: 1))/\mathrm{SL}_3 (\CC) \times \PP^9$ is rational, and the action of $\mathrm{PGL}_3 (\CC ) = \mathrm{SL}_3 (\CC )
/ Z$, where $Z$ is the center of $\mathrm{SL}_3 (\CC )$, is almost free on $\mathrm{Grass}(2,\: V(14,\: 1))$. Moreover, for the
$\mathrm{SL}_3 (\CC)$-linearized line bundle $\mathcal{O}_P (1)$ induced by the Pl\"ucker embedding
\[
\mathrm{Grass}(2,\: V(14,\: 1)) \subset \PP (\Lambda^2 (V(14,\: 1)))
\]
the locus of very stable points in the Grassmannian is then nonempty (one may choose $2$ linearly independent polynomial $\mathrm{SL}_3 (\CC
)$-invariants
$I$, $J$ on $V(14,\: 1)$ of the same degree and gets a nonvanishing polynomial invariant via $v\wedge w \mapsto I (v) J (w) -
J (v)  I(w)$ on the Grassmannian). Thus $\mathcal{O}_P(3)=: \mathcal{M}$ is $\mathrm{PGL}_3 (\CC )$-linearized on $\mathrm{Grass}(2,\:
V(14,\: 1))$, and if we choose on $\PP (V(0,\: 34)) \times \mathrm{Grass}(2,\: V(14,\: 1))$ the bundle $\mathcal{L}:=\mathcal{O}(1)\boxtimes
\mathcal{O}_P (2)$, all the assumptions of Proposition \ref{pFibrationOverGrass} except the dominance of $\varphi$ have been checked. The
latter dominance follows from an explicit computer calculation, as follows:

Choose a random point $x_0\in V(0, \: 34)$. If the map 
$$\psi(x_0, \cdot)\, :\, V(14,\: 1) \to
V(0,\: 21)$$
has maximal rank $253$, $\Psi$ is well defined. In this case compute a basis $y_1$, $y_2$ of $\ker \psi(x_0, \cdot)$.
Compute then the two
$253\times 630$-matrices $M_1$ resp. $M_2$ representing $\psi (\cdot, \: y_1)$ resp. $\psi (\cdot, \: y_2)$. If
\[
M:=\left( \begin{array}{c}
M_1\\ M_2
\end{array}\right) ,
\]
which is a $506\times 630$ matrix, has maximal rank $506= 2\cdot 253$, the kernel of $M$ represents the fibre $\Psi^{-1}(\Psi([x_0]))$ and is of expected dimension. 
Again one can easily do this calculation over a finite field using a computer algebra program. See \cite{script} for a Macaulay2 script.
\end{proof}

\begin{remark}\xlabel{rNoAlternatives}
As far as we can see, the rationality of $\PP (V(0,\: 34))/\mathrm{SL}_3 (\CC )$ cannot be obtained by direct application of Proposition
\ref{pFibrationOverGrass} with base of the projection a projective space. In fact, a computer search yields that the inclusion $V(0, \: 34)
\subset
\mathrm{Hom}(V(30, 0), \: V(0,4)
\oplus V(5,\: 9))$ is the only candidate to be taken into consideration for dimension reasons: $\dim V(30, 0) = \dim ( V(0,\: 4) \oplus V(5, \:
9) ) +1$ and $\dim \PP (V (0,\: 34)) > \dim \PP (V(30, \: 0))$. However, on $\PP (V(0, \: 34)) \times \PP (V (30, \: 0))$ there does not exist
a $\mathrm{PGL}_3 (\CC )$-linearized line bundle cutting out $\mathcal{O}(1)$ on the fibres of the projection to $\PP (V(30, 0))$; for such a
line bundle would have to be of the form $\mathcal{O}(1) \boxtimes \mathcal{O}(k)$, $k\in\ZZ $, and none of these is $\mathrm{PGL}_3 (\CC
)$-linearized: since
$\mathcal{O}\boxtimes \mathcal{O}(1)$ is $\mathrm{PGL}_3 (\CC )$-linearized it would follow that the $\mathrm{SL}_3 (\CC )$ action on $H^0(\PP
(V(0, \: 34), \: \mathcal{O}(1))\simeq V(34, \: 0)$ factors through $\mathrm{PGL}_3 (\CC )$ which is not the case.
\end{remark}

\end{document}